\documentclass[11pt,twoside]{article}
\usepackage{acta-m,euler}

\title{ Generalized M\"obius type functions and special set of $k$-free numbers
}

\newcommand{\vol}{\textbf{1}, 2 (2009) --}   %% to be completed by the editor, do not modify it
\newcommand{\amss}{\amssubj{
11A25, 11N37
}}   %% Source:   http://www.ams.org/msc/
\newcommand{\keyw}{\keywords{
M\"{o}bius function, generalized M\"obius function, k-free integers, generalized $k$-free integers
}}

\begin{document}
\maketitle

%% SINGLE AUTHOR. If you are a single author, please, use the following command and delete the \twoauthors commnad completely.

\oneauthor{
Antal Bege
%
 %% Write here the name of the author
}{
Sapientia--Hungarian University of Transilvania\\
Department of Mathematics and Informatics,\\
T\^argu Mure\c{s}, Romania
%
 %% Write here your affiliation including maybe your address. You can use \\ for line breaks.
}{abege@ms.sapientia.ro
%
 %% Write here your email.
}

%% TWO AUTHORS. If there are two authors, please, use the following command and and delete the \oneauthor commnad completely.

%\twoauthors{%
%% Write here the first author's name.
%}{%
 %% Write here the first author's affiliation including maybe his address. You can use \\ for line breaks.
%}{%
 %% Write here the first author's email.
%}{%
 %% Write here the second author's name.
%}{%
 %% Write here the second author's affiliation including maybe his address. You can use \\ for line breaks.
%}{%
 %% Write here the first author's email.
% }

%% MORE AUTHORS. For more than two authors, please, use both commands, maybe more than once.

%% Short name of the authors and short title, to be included in heading.

\short{A. Bege}{ Generalized M\"obius type functions}

\begin{abstract}
In \cite{bege1} Bege introduced the generalized Apostol's M\"obius
functions $\mu_{k,m}(n)$. In this paper we present new
properties of these functions. By  introducing the special set of
$k$-free numbers  we have obtained some asymptotic formulas for
the partial sums of these functions.

%% Write here the abstract of your paper.
\end{abstract}

\section{Introduction}

M\"obius function of order $k$, introduced by T. M. Apostol
\cite{apostol1}, is defined by the following formula:
$$
\mu_k(n)=
\left\{
\begin{array}{cl}
1& \mbox{ if }n=1,\\
0& \mbox{ if }p^{k+1}\mid n\mbox{ for some prime } p,\\
(-1)^r& \mbox{ if } n=p_1^k\cdots p_r^k\prod\limits_{i>r}p_i^{\alpha_i},
\quad \mbox{ with } 0\leq \alpha_i<k,\\
1& \mbox{ otherwise. }
\end{array}
\right.
$$
The generalized function is denoted by $\mu_{k,m}(n)$, where
$1<k\leq m$.
\\
If $m=k$, $\mu_{k,k}(n)$ is defined to be $\mu_k(n)$, and if $m>k$ the
function is defined as follows:
\begin{equation}
\label{21}
\mu_{k,m}(n)=
\left\{
\begin{array}{cl}
1& \mbox{ if } n=1,\\
1& \mbox{ if } p^{k}\nmid n\mbox{ for each prime } p,\\
(-1)^r& \mbox{ if } n=p_1^m\cdots p_r^m\prod\limits_{i>r}p_i^{\alpha_i},
\quad \mbox{ with } 0\leq \alpha_i<k,\\
0& \mbox{ otherwise. }
\end{array}
\right.
\end{equation}

In this paper we show some relations that hold among the functions
$\mu_{k,m}(n)$. We introduce the new type of $k$-free
integers and we make a connection between generalized
M\"obius function and the characteristic function $q^*_{k,m}(n)$ of
these. We use these to derive an asymptotic formula for the
summatory function of $q^*_{k,m}(n)$.

\section{Basic lemmas}

The generalization $\mu_{k,m}$, like Apostol's $\mu_k(n)$, is a multiplicative function
of $n$, so it is determined by its values at the prime powers. We have
$$
\mu_k(p^{\alpha })=\left\{
\begin{array}{rcl}
1& \mbox{ if }& 0\leq \alpha <k,\\
-1& \mbox{ if }& \alpha =k,\\
0& \mbox{ if }& \alpha >k,
\end{array}
\right.
$$
whereas
\begin{equation}
\label{22}
\mu_{k,m}(p^{\alpha })=\left\{
\begin{array}{rcl}
1& \mbox{ if }& 0\leq \alpha <k,\\
0& \mbox{ if }& k\leq \alpha <m,\\
-1& \mbox{ if }& \alpha =m,\\
0& \mbox{ if }& \alpha >m,
\end{array}
\right.
\end{equation}

\noindent
In \cite{apostol1} Apostol obtained the asymptotic formula
\begin{equation}
\label{11}
\sum_{n\leq x}\mu_k(n)=A_kx+O(x^{\frac{1}{k}}\log x),
\end{equation}
where
$$
A_k=\prod_p\left(1-\frac{2}{p^k}+\frac{1}{p^{k+1}}\right).
$$
Later, Suryanarayana \cite{suryanarayana1} showed that, on the assumption of
the Riemann hypothesis, the error term in (\ref{11}) can be improved to
\begin{equation}
\label{12}
O\left(x^{\frac{4k}{4k^2+1}}\omega (x)\right),
\end{equation}
Where
$$
\omega (x)=\mbox{exp}\{A\log x (\log \log x)^{-1}\}
$$
for some positive constant $k$.
\\
In 2001 A. Bege \cite{bege1} proved the following asymptotic formulas.
\\
\\
\begin{lemma}[ \cite{bege1}, Theorem 3.1.]
%{\bf Lemma 2.1. ( \cite{bege1}, Theorem 3.1.)}
%{\it
For $x\geq 3$ and $m>k\geq 2$ we have
\begin{equation}
\label{31}
\sum_{
\begin{array}{c}
r\leq x\\
(r,n)=1
\end{array}}
\mu_{k,m}(r)=
\frac{xn^2\; \alpha_{k,m}}{\zeta(k)\psi_k(n)\alpha_{k,m}(n)}+
0\left(\theta(n)x^\frac{1}{k}\delta(x)\right).
\end{equation}
uniformly in $x$, $n$ and $k$, where
$\theta(n)$ the number of square-free divisors of $n$,
$$
\alpha_{k,m}=\prod_p\left(1-\frac{1}{p^{m-k+1}+p^{m-k+2}+\cdots +p^m}\right),
$$
$$
\alpha_{k,m}(n)=n\prod_{p|n}\left(1-
\frac{1}{p^{m-k+1}+p^{m-k+2}+\cdots +p^m}\right),
$$
$$
\psi_k(n)=n\prod_{p|n}\left(1+\frac{1}{p}+\cdots +\frac{1}{p^{k-1}}\right),
$$
and
$$
\delta _k(x)=\mbox{ exp }\{-A\; k^{-\frac{8}{5}}\log^{\frac{3}{5}}x\; 
(\log \log x)^{-\frac{1}{5}}\},\quad A>0.
$$
\end{lemma}

\begin{lemma}[\cite{bege1}, Theorem 3.2.]
\label{2.2.}
%{\bf Lemma 2.2. ( \cite{bege1}, Theorem 3.2.)}
%\\
%{\it 
If the Riemann hypothesis is true, then for $x\geq 3$ and $m>k\geq 2$
we have
\begin{equation}
\label{33}
\sum_{
\begin{array}{c}
r\leq x\\
(r,n)=1
\end{array}}
\mu_{k,m}(r)=
\frac{xn^2\; \alpha_{k,m}}{\zeta(k)\psi_k(n)\alpha_{k,m}(n)}+
0\left(\theta(n)x^{\frac{2}{2k+1}}\omega(x)\right).
\end{equation}
uniformly in $x$, $n$ and $k$.
\end{lemma}

\begin{lemma}[\cite{apostol2}]
\label{2.4.}
%\noindent
%{\bf Lemma 2.4.} (\cite{apostol2})
%{\it 
If $s>0$, $s\not= 1$, $x\geq 1$, then
\[
\sum_{n\leq x}\frac{1}{n^s}=\zeta (s)-\frac{1}{(s-1)x^{s-1}}+O\left(\frac{1}{x^s}\right).
\]
\end{lemma}

\section{Generalized $k$-free numbers}

Let $Q_k$ denote the set of $k$-free numbers and let $q_k(n)$ to
be the characteristic function of this set. Cohen \cite{cohen1}
introduced the $Q_k^*$ set, the set of positive integers $n$ with
the property that the multiplicity of each prime divisor of $n$ is
not a multiple of $k$. Let $q_k^*(n)$ be the characteristic
function of these integers.
\[
q_k^*(n)=
\left\{
\begin{array}{ll}
1,& \mbox{ if }n=1\\
1,& \mbox{ if }n=p_1^{\alpha_1}\ldots p_k^{\alpha_k},\; \alpha_i\not\equiv 0\pmod{k}\\
0,& \mbox{ otherwise}.
\end{array}
\right.
\]
We introduce the following special set of integers
\begin{eqnarray*}
Q_{k,m}:&=&\{n\mid n=n_1\cdot n_2,\; (n_1,n_2)=1,\; n_1\in Q_k,\\
& &n_2=1\mbox{ or }n_2=(p_1\ldots p_i)^m,\; p_i\in \mathbb{P}\},
\end{eqnarray*}
with the characteristic function
\[
q_{k,m}(n)=
\left\{
\begin{array}{ll}
1,& \mbox{ if }n\in Q_{k,m}\\
0,& \mbox{ if }n\not\in Q_{k,m}.
\end{array}
\right.
\]
The function $q_{k,m}(n)$ is multiplicative and
\begin{equation}
\label{21_27}
q_{k,m}(n)=|\mu_{k,m}(n)|.
\end{equation}

%1
We introduce the following set $Q_{k,m}^*$ which the generalization of $Q_k^*$. The  integer $n$ is in the set $Q_{k,m}^*,\; 1<k<m$ iff the power of each prime divisor of $n$ divided by $m$ has the remainder between 1 and $k-1$. The characteristic functions of these numbers is
\[
q_{k,m}^*(n)=
\left\{
\begin{array}{ll}
1,& \mbox{ if }n=p_1^{\alpha_1}\ldots p_k^{\alpha_k},\; \exists \ell: \; \ell m< \alpha_i<\ell m+k\\
0,& \mbox{ otherwise}.
\end{array}
\right.
\]
If we write the generating functions for this  functions we
have the following result.

\begin{theorem}
%\noindent
%{\bf Theorem 3.1}
%{\it 
If
$m\geq k$ and the series converges absolutely we have
\begin{eqnarray}
\sum_{n=1}^{\infty }\frac{\mu_{k,m}(n)}{n^s}&=&\zeta(s)\prod_p\left(1-\frac{1}{p^{ks}}-\frac{1}{p^{ms}}+\frac{1}{p^{(m+1)s}}\right),\\
\sum_{n=1}^{\infty }\frac{q^*_{k,m}(n)}{n^s}&=&\zeta(s)\zeta(ms)\prod_p\left(1-\frac{1}{p^{ks}}-\frac{1}{p^{ms}}+\frac{1}{p^{(m+1)s}}\right),
\\
\sum_{n=1}^{\infty }\frac{q_{k,m}(n)}{n^s}&=&\zeta(s)\prod_p\left(1-\frac{1}{p^{ks}}+\frac{1}{p^{ms}}-\frac{1}{p^{(m+1)s}}\right).
\end{eqnarray}
\end{theorem}

\begin{proof}
The function $\mu_{k,m}(n)$ multiplicative, when the series converges absolutely ($s>1$) we have
\begin{eqnarray*}
\sum_{n=1}^{\infty }\frac{\mu_{k,m}(n)}{n^s}&=&\prod_{p}\left(1+\frac{\mu_{k,m}(p)}{p^s}+\ldots +\frac{\mu_{k,m}(p^{\alpha })}{p^{\alpha s}}+\ldots\right)=\\
&=&\prod_{p}\left(1+\frac{1}{p^s}+\ldots +\frac{1}{p^{(k-1)s}}-\frac{1}{p^{m s}}\right)=\\
&=& \prod_{p}\frac{1}{\displaystyle 1-\frac{1}{p^s}}\prod_p\left(1-\frac{1}{p^{ks}}-\frac{1}{p^{ms}}+\frac{1}{p^{(m+1)s}}\right)=
\\
&=&\zeta(s)\prod_p\left(1-\frac{1}{p^{ks}}-\frac{1}{p^{ms}}+\frac{1}{p^{(m+1)s}}\right).
\end{eqnarray*}
In the similar way because $q^*_{k,m}(n)$ multiplicative we have:
\begin{eqnarray*}
\sum_{n=1}^{\infty }\frac{q^*_{k,m}(n)}{n^s}&=&\prod_{p}\left(1+\frac{q^*_{k,m}(p)}{p^s}+\ldots +\frac{q^*_{k,m}(p^{\alpha })}{p^{\alpha s}}+\ldots\right)=\\
&=&\prod_{p}\left(1+\left(\frac{1}{p^s}+\frac{1}{p^{2s}}+\ldots +\frac{1}{p^{(k-1)s}}\right)+\right.\\
&+&\left.\left(\frac{1}{p^{(m+1)s}}+\frac{1}{p^{(m+2)s}}\ldots +\frac{1}{p^{(m+k-1)s}}\right)+\ldots \right)=\\
&=& \prod_{p}\left(1+\left(\frac{1}{p^s}+\frac{1}{p^{2s}}+\ldots +\frac{1}{p^{(k-1)s}}\right) \left(1+\frac{1}{p^{ms}}+\frac{1}{p^{2ms}}+\ldots \right)\right)
\\
&=&\prod_{p}\left(1+\frac{\displaystyle\frac{1}{p^s}-\frac{1}{p^{ks}}}{\displaystyle 1-\frac{1}{p^s}}\frac{1}{\displaystyle 1-\frac{1}{p^{ms}}}\right)=\\
&=&\zeta(s)\zeta(ms)\prod_p\left(1-\frac{1}{p^{ks}}-\frac{1}{p^{ms}}+\frac{1}{p^{(m+1)s}}\right).
\end{eqnarray*}
Because $q_{k,m}(n)$ multiplicative and $q_{k,m}(n)=|\mu_{k,m}(n)|$ we have:
\begin{eqnarray*}
\sum_{n=1}^{\infty }\frac{q_{k,m}(n)}{n^s}&=&\prod_{p}\left(1+\frac{q_{k,m}(p)}{p^s}+\ldots +\frac{q_{k,m}(p^{\alpha })}{p^{\alpha s}}+\ldots\right)=\\
&=&\prod_{p}\left(1+\frac{1}{p^s}+\ldots +\frac{1}{p^{(k-1)s}}+\frac{1}{p^{m s}}\right)=\\
&=& \prod_{p}\frac{1}{\displaystyle 1-\frac{1}{p^s}}\prod_p\left(1-\frac{1}{p^{ks}}+\frac{1}{p^{ms}}-\frac{1}{p^{(m+1)s}}\right)=
\\
&=&\zeta(s)\prod_p\left(1-\frac{1}{p^{ks}}+\frac{1}{p^{ms}}-\frac{1}{p^{(m+1)s}}\right),
\end{eqnarray*}
%\hspace*{\fill}$\Box $
\end{proof}

\noindent
In the particular case when $m=k$ we have $\mu_{k,m}(n)=\mu_{k}(n)$, $q_{k,m}(n)=q_{k+1}(n)$ and
\begin{eqnarray*}
\label{21_31}
\sum_{n=1}^{\infty }\frac{\mu_{k}(n)}{n^s}&=&\zeta(s)\prod_p\left(1-\frac{2}{p^{ks}}+\frac{1}{p^{(k+1)s}}\right),\\
%\label{21_32}
%\sum_{n=1}^{\infty }\frac{q^*_{k}(n)}{n^s}&=&\zeta(s)\zeta(ks)\prod_p\left(1-\frac{1}{p^{ks}}-\frac{2}{p^{ks}}+\frac{1}{p^{(k+1)s}}\right),
%\\
\label{21_33}
\sum_{n=1}^{\infty }\frac{q_{k+1}(n)}{n^s}&=&\frac{\zeta(s)}{\zeta\big((k+1)s\big)}.
\end{eqnarray*}

\noindent
We have the following convolution type
formulas.

\begin{theorem}
\label{t3_2}
If $m\geq k$
\begin{eqnarray}
\label{21_28}
q^*_{k,m}(n)&=&\sum_{d^m\delta= n}\mu_{k,m}(\delta ),\\
\label{21_29}
\mu_{k,m}(n)&=&\sum_{d^m\delta= n}\mu(d)q^*_{k,m}(\delta ).
%\label{21_30}
%\sum_{d\mid n}\mu_{k,m}(d)&=&\sum_{d\mid n}q^*_{k,m}(d)q_m\left(\frac{n}{d}\right).
\end{eqnarray}
\end{theorem}

\begin{proof}
Because $q_{k,m}(n)$ and $\mu_{k,m}(n)$ are multiplicative, results that both sides of (\ref{21_28}) are multiplicative functions. Hence it is enough if we verify the identity for $n=p^{\alpha }$, a prime power.
\\
If $\alpha =\ell m+i$ and $0<i<k$
\begin{eqnarray*}
\sum_{d^m\delta = p^{\alpha }}\mu_{k,m}(\delta )&=&\mu_{k,m}(p^{\ell m+i})+\mu_{k,m}(p^{(\ell -1) m+i})+\ldots +\mu_{k,m}(p^{m+i})+\\
&+&\mu_{k,m}(p^{i})=1=q_{k,m}(p^{\alpha }).
\end{eqnarray*}
If $\alpha =\ell m+i$ and $k<i<m$
\begin{eqnarray*}
\sum_{d^m\delta = p^{\alpha }}\mu_{k,m}(\delta )&=&\mu_{k,m}(p^{\ell m+i})+\mu_{k,m}(p^{(\ell -1) m+i})+\ldots +\mu_{k,m}(p^{m+i})+\\
&+&\mu_{k,m}(p^{i})=0=q_{k,m}(p^{\alpha }).
\end{eqnarray*}
If $\alpha =\ell m$
\begin{eqnarray*}
\sum_{d^m\delta = p^{\alpha }}\mu_{k,m}(\delta )&=&\mu_{k,m}(p^{\ell m})+\mu_{k,m}(p^{(\ell -1) m})+\ldots +\mu_{k,m}(p^{m})+\mu_{k,m}(1)=\\
&=&-1+1=0=q_{k,m}(p^{\alpha }).
\end{eqnarray*}
The (\ref{21_29}) follows by M\"obius inversion formula.
\\
\end{proof}

\section{Asymptotic formulas}

\begin{theorem}
%{\bf Theorem 4.1}
%{\it
For $x\geq 3$ and $m>k\geq 2$ we have
\begin{equation}
\label{41}
\sum_{
\begin{array}{c}
r\leq x
\end{array}}
q^*_{k,m}(r)=
\frac{x \alpha_{k,m}\zeta(m)}{\zeta(k))}+
0\left(x^\frac{1}{k}\delta(x)\right).
\end{equation}
uniformly in $x$, $n$ and $k$, where
$$
\alpha_{k,m}=\prod_p\left(1-\frac{1}{p^{m-k+1}+p^{m-k+2}+\cdots +p^m}\right).
$$
$$
\delta (x)=\mbox{ exp }\{-A\; \log^{\frac{3}{5}}x\; 
(\log \log x)^{-\frac{1}{5}}\},
$$
for some absolute constant $A>0$.
\end{theorem}

\begin{proof}
By (\ref{21_28}) and (\ref{31}) with $n=1$ we have
\[
\sum_{
\begin{array}{c}
r\leq x
\end{array}}
q^*_{k,m}(n)=
\sum_{
\begin{array}{l}
\delta d^m\leq x
\end{array}}
\mu_{k,m}(\delta )
=
\sum_{
\begin{array}{l}
d\leq x^{\frac{1}{m}}
\end{array}}
\sum_{
\begin{array}{l}
\delta \leq \frac{x}{d^m}
\end{array}}
\mu_{k,m}(\delta )=
\]
\[
=
\sum_{
\begin{array}{l}
d\leq x^{\frac{1}{m}}
\end{array}}
\left\{
\frac{\left(\frac{x}{d^m}\right)\alpha_{k,m}}{\zeta (k)}+
0\left(\frac{x^{\frac{1}{k}}}{d^{\frac{m}{k}}}
\delta \left(\frac{x}{d^m}\right)
\right)
\right\}=
\]
\[
=
\frac{x\alpha_{k,m}}{\zeta(k)}
\sum_{
\begin{array}{c}
d\leq x^{\frac{1}{m}}
\end{array}}
\frac{1}{d^{m}}
+
O\left(
\delta(x)x^{\epsilon }
x^{\frac{1}{k}-\epsilon }
\sum_{
\begin{array}{c}
d\leq x^{\frac{1}{m}}
\end{array}}
\frac{1}
{d^{\frac{m}{k}-\epsilon m}}
\right)
\]
Now we use \ref{2.4.} and the fact that $\delta(x)x^{\epsilon }$ is
increasing for all $\epsilon >0$, then we choose $\epsilon >0$ so
that $\frac{m}{k}-\epsilon m>1+\epsilon '$ we obtain (\ref{41}).
\end{proof}

\noindent
Applying the method used to prove Theorem 1, and making use of
Lemma \ref{2.2.} we get to

\begin{theorem}
%{\bf Theorem 4.2}
%\\
%{\it 
If the Riemann hypothesis is true, then for $x\geq 3$ and $m>k\geq 2$
we have
\begin{equation}
\sum_{
\begin{array}{c}
r\leq x
\end{array}}
q^*_{k,m}(r)=
\frac{x \alpha_{k,m}\zeta(m)}{\zeta(k)}+
0\left(x^{\frac{2}{2k+1}}\omega(x)\right).
\end{equation}
uniformly in $x$, $n$ and $k$.
\end{theorem}

%% Write here your paper using the section command, theorem-like environments as:
%% definition, theorem, lemma, corollary, criterion, example, exercise, notation, problem, proposition, remark.
%% For proof you can use the proof environment.

%% For references use the following format. For citation use e. g. \cite{lothaire}.

\bigskip
\rightline{\emph{Received: May 3, 2009}}     %% to be completed by the editor

\newpage

\ \

\thispagestyle{empty}

\end{document}